\documentclass[12pt,a4paper]{article} %ГришаТамасян
\usepackage[cp866]{inputenc}
\usepackage[english]{babel}
\usepackage{amsmath}
\usepackage{amsfonts}
\usepackage{longtable}
\usepackage{amssymb}
\usepackage{graphics}
\newcommand{\bc}{\begin{center}}
\newcommand{\br}{\begin{right}}
\newcommand{\ec}{\end{center}}
\newcommand{\be}{\begin{equation}}
\newcommand{\ee}{\end{equation}}

\newcommand{\vl}{\mid}

\newcommand{\rar}{\rightarrow}
\newcommand{\grad}{\nabla}

\newcommand{\p}{\partial}

\newcommand{\bco}{\mbox{co} \,}

\newcommand{\blim}{\mbox{lim} \,}
\newcommand{\bsup}{\mbox{sup} \,}

\newcommand{\binf}{\mbox{inf} \,}
\newcommand{\exs}{\exists}

\newcommand{\meq}{\geq}

\newcommand{\la}{\lambda}
\newcommand{\La}{\Lambda}
\newcommand{\al}{\alpha}

\newcommand{\del}{\delta}

\newcommand{\eps}{\varepsilon}

\newtheorem{thm}{Theorem}[section]

\newtheorem{rem}{Remark}[section]
\newtheorem{defi}{Definition}[section]
\newtheorem{cor}{Corollary}[section]

\newtheorem{conseq}{Consequence}[section]

\begin{document}

AMS 517.977

\vspace{1cm}

{\bf Prudnikov I.M.}

\vspace{0.5cm}

\begin{center}
{\bf Upper and lower approximations of the second order.
Exhausters of the second order}
\end{center}

\vspace{1cm}

The upper and lower approximations of the second order are defined
for a function differentiable in the directions. Exhausters of the
second order are introduced. The relation between the second-order
Frechet subdifferentials and the second-order exhausters is given.
A rule for constructing exhausters of the first and second orders
is given for Lipschitz functions.Necessary and sufficient extremum
conditions through exhausters are given. It is discussed  about
the use of Exhausters in optimization.

\vspace{0.5cm}


\vspace{0.5cm}

\section{\bf Introduction}

\vspace{1cm}

We will assume further  that $f(\cdot):\mathbb{R}^{n} \to
\mathbb{R}$ is an arbitrary continuous, directional differentiable
function. There are various definitions of directional
differentiability.

Let's introduce some notation
$$
f_D^{\uparrow }(x_{0},g)=  \blim \bsup_{ \al \to +0}
(f(x_{0}+\alpha g ) - f(x_{0}))/\alpha .
$$
$f_D^{\uparrow}(x_{0}, g)$ is called the upper Dini derivative of
$f(\cdot)$ at the point $x_{0}$ in the direction $g$. The limit
$$
f_D^{\downarrow }(x_{0},g)=  \blim \binf_{ \al \to +0}
(f(x_{0}+\alpha g ) - f(x_{0}))/\alpha
$$
is called the lower Dini derivative at $x_0$ in the direction $g$.

If $f_D^{\uparrow}(x_{0},g)=f_D^{\downarrow }(x_{0},g),$ for all $
g \in \mathbb{R}^n$ then the function $f(\cdot)$ is called Dini
differentiable at the point $x_0$.

The upper Hadamard derivative  $f_H^{\uparrow } (x_0, g)$ of
$f(\cdot)$ at the point $x_0$ in the direction $g \in
\mathbb{R}^n$ is called the limit
$$
f_H^{\uparrow } (x_0, g)= \lim \sup_{\footnotesize \begin{array}{cc} g' \rar g \\
\al \rar 0^+ \end{array}} \frac{f(x_0 +\al g') - f(x_0)}{\al}.
$$
The limit
$$
f_H^{\downarrow } (x_0, g)= \lim \inf_{\footnotesize \begin{array}{cc} g' \rar g \\
\al \rar 0^+
\end{array}} \frac{f(x_0 +\al g') - f(x_0)}{\al},
$$
is called the lower Hadamard derivative of $f(\cdot)$ at $x_0$ in
the direction $g$. If $f_H^{\uparrow}(x_{0},g)=f_H^{\downarrow
}(x_{0},g)$ for all $g \in \mathbb{R}^n $, then the function
$f(\cdot)$ is called Hadamard differentiable at $x_0$.

We will assume that $f(\cdot)$ is directional differentiable in
the Dini sense. Note that for a Lipschitz  function $f(\cdot)$ in
a neighborhood of the point $x_0$ the equality
$$
f'_H (x_0, g) = f'_D (x_0, g),
$$
holds, as soon as
$$
\mid f(x_0 +\al g') - f(x_0+\al g) \mid  \leq \al L \|g' - g \|.
$$
Also we can state in this case that    $f_D(x_0, \cdot)$ is finite
and Lipschitz in $g \in \mathbb{R}^{n}$. The Dini derivative of
$f(\cdot)$ in the direction $g$ at $x_0$ will be denoted by
$f'(x_0, g)$.

Academician B.N. Pshenichnyi introduced the concept of upper
convex approximation (UCA) of a function $f(\cdot)$ at a point
$x_0$.

\begin{defi}  \cite{pshenichnyi} A function $h(\cdot):
\mathbb{R}^n \rar \mathbb{R}$ is called  UCA of the first-order of
a function $f(\cdot)$ at a point $x_0$ if

1) $h(g) \ge f'(x_{0},g) \,\,\,\, \forall g \in \mathbb{R}^n,$

2) $h(\cdot)$ is a convex closed positively homogeneous (p.h.)
function of first order.

\end{defi}

Lower concave approximations (LCA) of the function $f(\cdot)$ at
$x_0$ were introduced in \cite{demrub}.

\begin{defi}  A function $\chi(\cdot): \mathbb{R}^n \rar
\mathbb{R}$ is called LCA  of the first order of a function
$f(\cdot)$ at a point $x_0$ if

1) $\chi(g) \leq f'(x_{0},g) \,\,\,\, \forall g \in \mathbb{R}^n,$

2) $\chi(\cdot)$ is a  concave closed p.h. function of the first
order.
\end{defi}

Obviously, there are infinitely many UCA and LCA. Using UCA and
LCA, one can formulate the necessary and sufficient conditions for
minimum or maximum of a function at a point \cite{demroshina},
\cite{demyanovnexhauster}.

Academician B.N. Pshenichnyi defined the main upper convex
approximations \cite{pshenichnyi}, p. 221, based on which
Professor V.F. Demyanov introduced the concept of an exhaustive
set of UCA: $h_{\lambda}$, $\lambda \in \Lambda$, where $\Lambda$
is a finite or infinite set of indices \cite{demyanovnexhauster},
\cite{demyanovnexhauster2}. This is a set of UCA $h_{\lambda}$,
$\lambda \in \Lambda$ such that
$$
     \inf_{\la \in \La} h_{\la} (g) = f'(x_0, g)  \,\,\,\, \forall g
     \in S^{n-1}_1(0)=\{ v \in \mathbb{R}^n \vl \| v \| =1 \}.
$$
The exhaustive set of UCA  is also called an upper Exhauster  of
$f(\cdot)$ at $x_0$ and is denoted by $E^*(f)$. We will also call
it the upper Exhauster of the first order and denote it by
$E_1^*(f)$. The set of functions from $E_1^*(f)$ forms the upper
exhaustive set $E_1^*(h_1)$ of $h_1(\cdot)$ at 0, where
$h_1(g)=f'(x_0,g)$.

Similarly, the lower exhaustive set of the first order $E_{1*}(f)$
is introduced as such a set consisting of $\chi_{\mu}(\cdot)$,
$\mu \in \Upsilon $, for which
$$
\sup_{\mu \in \Upsilon} \chi_{\mu}(g) = f'(x_0,g) \,\,\,\, \forall
g \in S^{n-1}_1 (0).
$$
Biexhauster is a pair $[E_1^*(f), E_{1*}(f)] $.

The construction of the upper and lower Exhausters is interesting
for optimization, since they can be used to formulate necessary
and sufficient conditions for the minimum and maximum of a
function $f(\cdot)$ at a point $x_0$ \cite{demroshina},
\cite{proudexhaust}. In \cite{pshenichnyi}, rules for constructing
upper and lower Exhausters were given for various combinations of
functions (sum, maximum, minimum). If we know the upper and lower
Exhausters of a function $f(\cdot)$ at a point $x$, then we can
find the directions of fastest increase and decrease of $f(\cdot)$
at a point $x$. This is the essence of constructing Exhausters.
The construction of the first and second order Exhausters is
closely related to  calculation of quasidifferentials and
codifferentials for quasidifferentiable and codifferentiable
functions. This is also important for optimization.

Since, according to Minkowski duality, every convex finite p.h.
function $h(\cdot)$ is uniquely determined by its subdifferential
at zero $\p h(0)$, we can assume that the Exhauster $E^*_1(f)$
consists of convex compact sets $\p h_{\la}(0)$, $\la \in \La $,
which are subdifferentials at zero of some convex p.h. functions
$h_{\la}(\cdot) \in E^*_1(f)$, $\la \in \La$.

%%%%%%%%%%%%%%%%%%%%%%%%%%%%%%%%%%%%%%%%%%%%%%%%%%%%

\begin{thm} \cite{demroshina}
Let $ E^*_1(f)$ be a first-order upper Exhauster of a function
$f(\cdot)$ at $x_0$. Then, if for any $C \in E^*_1(f) $ there
exists $ \delta>0 $ such that
$$
B^n_{\del} \subset C,
$$
where $B^n_{\del}(0) = \{ z \in \mathbb{R}^n \vl \| z \| \leq \del
\} $, then $x_0$ is a minimum.
\end{thm}

We can also formulate a sufficient condition for a maximum of the
function $f(\cdot)$ at $x_0$ using the lower Exhauster $ E_{1*}(f)
$.

Rubinov A.M. proved {\cite{demrub}, p. 171} that if a function
$h_1(\cdot): \mathbb{R}^n \rar \mathbb{R}$, \,\,\, $h_1(g)=f'(x_0,
g)$ is upper semicontinuous (u.s.c.) in $g$, that is
$$
\limsup_{y \rar g} f'(x_0, y) \leq f'(x_0,g) \,\,\, \forall g \in
S^{n-1}_1(0)
$$
and is bounded above on $B^n_1(0)$, then there exists an upper
Exhauster $E_1^*(f)$ of $f(\cdot)$ at the point $x_0$, and hence
an upper Exhauster $E_1^*(h_1)$ of $h_1(\cdot)$ at zero.
Similarly, if $h_1(\cdot)$ is lower semicontinuous (l.s.c.) in
$g$, that is,
$$
\liminf_{y \rar g} f'(x_0, y) \meq f'(x_0,g) \,\,\, \forall g \in
S^{n-1}_1(0)
$$
and is bounded below on $B^n_1(0)$, then there exists a lower
Exhauster $E_{1*}(f)$ of $f(\cdot)$ at $x_0$ consisting of a
family of superlinear functions $q_{\la}$ such that
$$
\sup_{\la \in \La} q_{\la} (g) = f'(x_0, g) \,\,\,\, \forall g \in
S^{n-1}_1(0)=\{ v \in \mathbb{R}^n \vl \| v\| =1\}.
$$
%%%%%%%%%%%%%%%%%%%%%%%%%%%%%%%%%%%%%%%%%%%%%%%%%%%%%%%%%

\begin{thm}
If a function $f(\cdot)$ is u.s.c.  at $x_0$, that is
$$
\limsup_{x \rar x_0} f(x) \leq f(x_0),
$$
and uniformly differentiable with respect to directions at $x_0$,
then the function $h_1(\cdot):$ $h_1(g)=f'(x_0,g)$ is also u.s.c.
with respect to $g$.
\end{thm}
\begin{rem}
Any function $f(\cdot)$ that is Lipschitz  in a neighborhood of
$x_0$ and differentiable with respect to directions is uniformly
differentiable with respect to directions at $x_0$. Therefore, the
formulated theorem is true for the Lipschitz function $f(\cdot)$.
\end{rem}

{\bf Proof}. Let us carry out a series of calculations
$$
\limsup_{y \rar g} h_1(g)=\limsup_{y \rar g} \lim_{\al \rar 0^+}
\frac{f(x_0+\al \, y)-f(x_0)}{\al} =
$$
$$
= \lim_{\al \rar 0^+} \limsup_{y \rar g} \frac{f(x_0+\al \,
y)-f(x_0)}{\al} \leq \lim_{\al \rar 0^+} \frac{f(x_0+\al \,
g)-f(x_0)}{\al} = h_1(g),
$$
i.e. $h_1(\cdot)$ is u.s.c. at the point $g$, which was required
to be proved. $\Box$

B.N. Pshenichnyi also proved {\cite{pshenichnyi}, p. 222} that if
$h_1(\cdot)$ takes finite values and is also closed, i.e. its
epigraph $ epi f = \{ (x,y) \in \mathbb{R}^{n+1} | y \meq h_1(x)
\}$ is a closed set, which is equivalent to the lower
semicontinuity of $h_1(\cdot)$ \cite{pshenichnyi}, p. 62, i.e.
$$
\lim \inf_{g_k \rar g_0} h_1(g_k) \meq h_1(g_0),
$$
then there exist a principal UCA $h_{1\lambda}(\cdot) $ such that
$$ h_1(g)= \inf_{\lambda} h_{1\lambda}(g).$$ In the terminology of
Demyanov V.F. and Rubinov A.M., the latter means that there exists
an exhaustive set of UCA. \cite{demrub}

%%%%%%%%%%%%%%%%%%%%%%%%%%%%%%%%%%%%%%%%%%%%%%%%%%%%%%%%%%%

\begin{thm}
If $h_1(\cdot)$ is lower semicontinuous (l.s.c.) and does not take
infinite values, then there exist some lower $E_{1*}(h_1)$ and
upper $E_{1}^*(h_1)$ Exhausters of $h_1$.

If $h_1(\cdot)$ is upper semicontinuous (u.s.c.) and does not take
infinite values, then there exist some lower $E_{1*}(h_1)$ and
upper $E_{1}^*(h_1)$ Exhausters of $h_1$. \label{thmexh0}
\end{thm}
{\bf Proof.} The first assertion follows from the results of A.M.
Rubinov \cite{demrub}, p. 171, and B.N. Pshenichnyi
\cite{pshenichnyi}, p. 222. We prove the second assertion.

Let $h_1(\cdot)$ be u.s.c. that is
$$
\lim \sup_{g_k \rar g} h_1(g_k) \leq h_1(g).
$$
According to  Rubinov's theorem, there exists an upper Exhauster
$E_{1}^*(h_1)$, i.e. a family of positively homogeneous convex
functions $h_{\lambda}(\cdot)$, $\lambda \in \Lambda$, such that
$$
h_1(g)= \inf_{\lambda \in \Lambda} \, h_{1\lambda}(g) \,\,\,
\forall g \in \mathbb{R}^n.
$$
We multiply both sides of the last inequality by $-1$. Then
$$
-\lim \sup_{g_k \rar g} h_1(g_k) \geq -h_1(g)
$$
or, taking into account that
$$
\sup_{g} h_{1 } (g) = - \inf_{g}( - h_{1} (g)),
$$
we have
$$
- \lim - (\inf_{g_k \rar g} -h_1 (g_k)) = \lim \inf_{g_k \rar g} -
h_1(g_k) \geq -h_1(g) \,\,\, \forall g \in \mathbb{R}^n,
$$
i.e. the function $-h_1(\cdot) $ is lower semicontinuous.
According to the theorem of Pshenichnyi B.N. there exists a family
of convex positively homogeneous functions $-q_{\lambda }(\cdot)$,
$ \lambda \in \Lambda$, such that
$$
-h_1(g) = \inf_{\lambda \in \Lambda} -q_{\lambda}(g) \,\,\,
\forall g \in \mathbb{R}^n.
$$
But then
$$
h_1(g)= \sup_{\lambda \in \Lambda} q_{\lambda}(g) \,\,\, \forall g
\in \mathbb{R}^n.
$$
Furthermore, if $-q_{\lambda}(\cdot) $ is a convex positively
homogeneous function, then $q_{\lambda}(\cdot)$ is a concave
positively homogeneous function. It follows that $h_1(\cdot)$ has
a lower Exhauster. The theorem is proved. $\Box$

%%%%%%%%%%%%%%%%%%%%%%%%%%%%%%%%%%%%%%%%%%%%%%%%%%%%%%%%%%%%%

\section{Construction of the Exhausters}

The construction of $E_1^*(f)$ and $E_{1*}(f)$ for a Lipschitz
function $f(\cdot)$ with Lipschitz constant $L$ is described in
\cite{proudexhaust}, where the method of cutting planes
(hyperplanes) is presented. To do this, we first construct the
function
$$
\tilde f (x) = f(x) + L \| x - x_0 \|.
$$
Next, we determine the sets of averaged integral gradients,
calculated along some curves $r(x_0, \cdot, g) \in \eta(x_0)$.
Here $ \eta(x_0)$ is a set of the curves $r(x_0, \cdot, g) = x_0
+\al \, g + o(\al)$, $g \in S^{n-1}_1(0)$ defined in
\cite{pimnewconstr1}.

\begin{defi}. Let $\eta (x_0)$ be the set of curves
$r(x_0,\al,g) = x_0 + \al g + o_r(\al)$, where $g \in
S^{n-1}_1(0)$ and the function $o_r(\cdot):[0,\al_0] \rar
\mathbb{R}^n, \al_0>0$
satisfies the following conditions \\
1) $o_r(\al)/(\al) \rar +0$ for $\al \rar +0$ uniformly
for all curves $r(\cdot)$ ; \\
2) the function $o_r(\cdot)$ is continuously differentiable and
its derivative $o'(\cdot)$ is bounded above in norm near the
origin: there exists $c<\infty$ such that
$$
\max_{\tau \in [0,\al_0]} \parallel o_r'(\tau) \parallel \leq c
$$
3) the gradient $\grad f(r(x_0,\cdot, g))$ exists almost
everywhere (AE) along the curve $r(x_0,\cdot ,g)$.
\end{defi}

\begin{rem} By property 3 of the definition, the set
$\eta(x_0)$ depends on the choice of the function $f(\cdot)$. The
constants $c$ and $\al_0$ are the same for all curves $r \in
\eta(x_0)$.
\end{rem}

We introduce the following set of vectors
$$
E f(x_0) = \Bigl \{ v \in \mathbb{R}^n : \exs {\al_{k}}, \al_{k}
\rar +0, ( \exs \,g \in S^{n-1}_1 (0) ),
$$
$$(\exs r(x_0,\cdot ,g)\, \in \, \eta(x_0)) ,
v = \lim_{\al_k \rar +0} \al_k^{-1} \; \int^{\al_k}_0 \, \grad
f(r(x_0,\tau,g))d\tau \Bigr \}, \,\, Df(x_0)= \mbox{co} \,\,
Ef(x_0),
$$
where $\grad f$ is the gradient of $f$ at the points where it
exists, and the integral is understood in the Lebesgue sense. From
here and on, we use the notations $S^{n-1}_1(0)$ and $B^n_1(0)$
for the sphere and ball of radius 1 in n-dimensional space
centered at the origin. The convexity of the set $Df(x_0)$ is
obvious. The closure property is proved, and a connection is
established between the set $Df(x_0)$ and the Clarke
subdifferential $\p_{CL}f(x_0)$ \cite{pimnewconstr1} for any
Lipschitz function $f(\cdot)$.

Take any vector $g \in S^{n-1}_1 (0)$. Let
$$
P_g (v)=\{ w \in \mathbb{R}^n \vl (w,g) \leq (v,g) \}.
$$
Here $(a,b)$ denotes the scalar product of vectors $a$ and $b$.
Obviously, $ P_g (v)$ is a half-space for any vector $v$. For any
vector $v \in E_g(\tilde{f})$ we define
$$
C_g(\tilde{f}) = D\tilde{f}(x_0) \cap P_g (v),
$$
Where
$$
E_g \tilde{f}(x_0) = \bco \{ v \in \mathbb{R}^n : \exs {\al_{k}},
\al_{k} \rar +0,
$$
$$
(\exs r(x_0,\cdot ,g)\, \in \, \eta(x_0)) , v = \lim_{\al_k \rar
+0} \al_k^{-1} \; \int^{\al_k}_0\, \grad
\tilde{f}(r(x_0,\tau,g))d\tau \; \}.
$$
Obviously, $C_g(\tilde{f})$ is a convex compact set for any $g \in
S^{n-1}_1(0)$.

%%%%%%%%%%%%%%%%%%%%%%%%%%%%%%%%%%%%%%%%%%%%%%%%%%%%%%%%%%%%%%

The following theorems are proved in \cite{proudexhaust}.

\begin{thm}
A bounded family of convex compact sets $C$, including the sets
$C_g(\tilde{f})$ for all $g \in S^{n-1}_1(0)$, forms the upper
Exhauster $E^*(\tilde{f})$, i.e. the equality
$$
\tilde f'(x_0,g) \overset{\mathrm{def}}{=} \tilde{h}(g)= \inf_{C
\in E^*(\tilde{f})} \,\, \tilde{h}_C (g),
$$
where $\tilde{h}_C (g)=\max_{v \in C } \, (v,g)$, holds.
\label{thmexh1}
\end{thm}

\begin{rem} The theorem provides a rule for constructing the upper Exhauster
$E^*(\tilde{f})$. Clearly, the upper Exhauster $E^*(\tilde{f})$ of
the function $\tilde{f}$ at $x_0$ coincides with the upper
Exhauster of the function $\tilde{h}(\cdot)$ at zero, i.e.,
$E^*(\tilde{f}) = E^*(\tilde{h}).$
\end{rem}

Knowing $E^*_1(\tilde f)$, we can find the first-order upper
Exhauster of the function $f(\cdot)$ at $x_0$. The following
theorem is proved in \cite{proudexhaust}.

\begin{thm} The upper
Exhauster $E^*({f})$ of ${f}(\cdot)$ is defined by the equality
$$
E^*({f})= \{w + C \vl w \in L B^n_1(0), C \in E^*(\tilde{f})\},
$$
where $ E^*(\tilde{f})$ is the upper Exhauster of
$\tilde{f}(\cdot)$ at $x_0$, which we can construct according to
Theorem \ref{thmexh1}. \label{thmexh2}
\end{thm}

Now we construct a lower Exhauster $E_*({f})$ at $x_0$ of the
Lipschitz function ${f}(\cdot)$. We define the function
$\bar{h}(\cdot):\mathbb{R}^n \rar \mathbb{R}$
$$
\bar{h}(g)= {h}(g)-L \| g \|.
$$
We denote by $E^*(-\bar{h})$ the upper Exhauster of the function
$-\bar{h}(\cdot)$. We already know that
$$
-\bar{h}(g)=\min_{C \in E^*(-\bar{h})} \max_{v \in C} (v,g)
$$
or
$$
\bar{h}(g)=-\min_{C \in E^*(-\bar{h})} \max_{v \in C} (v,g)=
\max_{C \in E^*(-\bar{h})} [-\max_{v \in C} (v,g)]=
$$
$$
=\max_{C \in E^*(-\bar{h})} \min_{v \in (- C)} (v,g).
$$
We have
$$
{h}(g)= \bar{h}(g) + \max_{v \in LB^n_1(0)}(v,g)= \max_{C \in
E^*(-\bar{h})} \min_{v \in (- C)} (v,g) + \max_{w \in
LB^n_1(0)}(w,g)=
$$
$$
=\max_{C \in E^*(-\bar{h})} [\max_{w \in LB^n_1(0)} \min_{v \in
w+(-C)} (v,g) ]= \max_{ C \in [ -E^*(-\bar{h})+w \vl w \in
LB^n_1(0) ]} \min_{v \in C} (v,g).
$$
From this we obtain the following theorem.
\begin{thm}
The lower Exhauster $E_*({f})$ of the Lipschitz function
${f}(\cdot)$ at $x_0$ is the set
$$
E_*({f})= \{w + C \vl w \in L B^n_1(0), C \in -E^*(-\bar{h}) \},
$$
where $ E^*(-\bar{h})$ is the upper Exhauster of the function
$-\bar{h}(\cdot)$ at zero, which we can construct according to
Theorem \ref{thmexh1}. \label{thmexh3}
\end{thm}
\begin{cor}
By definition, the Biexhauster of the Lipschitz function
${f}(\cdot)$ at the point $x_0$ is a pair $[A,B]$ of the sets $A$,
$B$, where
$$
A= \{ w + C \vl w \in L B^n_1(0), C \in E^*(\tilde{h}) \}, \, B=
\{ w + C \vl w \in L B^n_1(0), C \in -E^*(-\bar{h}) \}
$$
\end{cor}

%%%%%%%%%%%%%%%%%%%%%%%%%%%%%%%%%%%%%%%%%%%%%%%%%%%%%%%%%%%%

\section{The lower Frechet subdifferential of the second order. The upper second-order
Exhauster}

The lower Frechet subdifferential  of the first order of a
function $f(\cdot)$ at a point $x_0$ is defined as follows
\cite{ioffepenot}
$$
\p_F^- f(x_0)=\{ v \in \mathbb{R}^n \vl \liminf_{\| g \| \rar 0}
\frac{1}{\| g \|} (f(x_0 + g)-f(x_0)-(v,g)) \meq 0 \}.
$$
Here $(v, g)$ is the scalar product of vectors $v$ and $g$.

There is a relationship between the lower Frechet subdifferential
$\p_F^- f(x_0)$ and the first-order upper Exhauster $E_1^*(f)$ at
the point $x_0$ \cite{roshchina}, \cite{demroshina}. This
relationship is given by the relation \be \bigcap_{C \subset
E_1^*(f)} C = \p_F^- f(x_0). \label{exsecond1} \ee

The relationship between the upper Frechet subdifferential $\p_F^+
f(x_0)$ and the first-order lower Exhauster is written as follows
\cite{roshchina}, \cite{demroshina}

\be \bigcap_{C \subset E_{1*}(f)} C = \p_F^+ f(x_0).
\label{exsecond2} \ee

The equalities (\ref{exsecond1}) and (\ref{exsecond2}) were given
in \cite{roshchina}.

Some function $f(\cdot)$ is called Frechet subdifferentiable if
$\p_F^- f(x_0)$ is nonempty. Also, $f(\cdot)$ is called Frechet
superdifferentiable if $\p_F^+ f(x_0)$ is nonempty. It is known
that the sets $\p_F^- f(x_0)$ and $\p_F^+ f(x_0)$ are nonempty if
and only if $f(\cdot)$ is Frechet differentiable at $x_0$ and then
$\p_F^- f(x_0)=\p_F^+ f(x_0)=\{ f' (x_0) \}$ \cite{kruger}.

The second-order Frechet lower subdifferential was introduced in
\cite{penot}, \cite{ioffepenot} \be \p_-^2 f(x_0)=\{ [v,A] \vl
\liminf_{\| g \| \rar 0} \frac{1}{\| g \|^2} (f(x_0+
g)-f(x_0)-(v,g)-\frac{1}{2}(Ag,g)) \meq 0 \}. \label{exsecond3}
\ee It is clear from the definition that $ \p_-^2 f(x_0)$ is a
convex closed set.

We define the function $h^+_{\la}(\cdot): \mathbb{R}^n \rar
\mathbb{R}$
$$
h_{\la}^+(g)=\max_{[ v,A] \in \p^2 h^+_{\la}(0)} [(v,g)+
\frac{1}{2}(Ag,g) ],
$$
where $\p^2 h^+_{\la}(0) \subset \mathbb{R}^n \times
\mathbb{R}^{n^2}$ is a convex compact set called the second
subdifferential of the function $h_{\la}^+(\cdot)$ at zero. Here
$(v,g)$ is the scalar product of vectors $v$ and $g$.

%%%%%%%%%%%%%%%%%%%%%%%%%%%%%%%%%%%%%%%%%%%%%%%%%%%%%%%%%%%%%%
We say that $h_{\la}^+(\cdot)$ is a second-order upper
approximation of the function $f(\cdot)$ at the point $x_0$ if
there exists $\al_0 >0$ such that \be f(x_0+\al g) - f(x_0) \leq
h_{\la}^+(\al g) + o_{\la}(\al^2 g^2) \,\,\, \forall g \in
S^1_1(0), \,\,\, \forall \al \in [0, \al_0], \al_0>0,
\label{exsecond4} \ee where

\be \lim_{\al \rar 0^+} \frac{o_{\la}(\al^2 g^2)}{\al^2} =0
\label{exsecond5} \ee uniformly in $g \in B^{n}_1(0)$, \, $g^2 =
\| g\|^2$.

We say that a set of functions $h_{\la}^+(\cdot)$, $\la \in \La$,
forms an upper second-order Exhauster $E^*_{2}(f)$ of the function
$f(\cdot)$ at the point $x_0$ if the inequality (\ref{exsecond4})
holds for any $\la \in \La$, where $\La$ is a finite or infinite
set of indices $\la$, and also \be f(x_0+\al g)- f(x_0) =
\inf_{\la \in \La} h^+_{\la} (\al g) + o (\al^2 g^2) \,\,\,\,
\forall g \in B^n_1(0), \label{exsecond5a} \ee where $ o(\cdot) $
satisfies conditions similar to those (\ref{exsecond5}).

Note that if the set of functions $h^+_{\la}(\cdot)$, $\la \in
\La$, forms the second-order upper Exhauster, then the set of
functions $h^+_{1\la}(\cdot):\mathbb{R}^n \longrightarrow
\mathbb{R}$, $\la \in \La$,
$$
h^+_{1\la}(g)=\max_{v \in \p h^+_{1\la}(0)} (v,g),
$$
where $ \p h^+_{1\la}(0)=\{ v \vl \exists A: (v,A) \in \p^2
h^+_{\la}(0) \}$, forms the first-order upper Exhauster
$E_1^*(h_1)$ of the function $h_1(g)=f'(x_0,g)$ at the point $0$.

Our goal will be to establish a connection between the second
subdifferentials $\p^2 h^+_{\la}(0)$ of the functions
$h_{\la}^+(\cdot)$, $\la \in \La$, and the second-order lower
Frechet subdifferential $\p^2_- f(x_0)$.

Everywhere below, we assume that $f(\cdot)$ has the finite second
derivative $f''(x_0, g, g)$ in the directions $g \in \mathbb{R}^n$
in the sense that
$$
\sup_{\| g \| \leq 1} \vl f''(x_0, g, g) \vl < \infty
$$
and the equality
$$
f(x_0 + \al g)= f(x_0) +\al f'(x_0,g)+ \frac{\al^2}{2}f''(x_0, g,
g)+ o(\al^2 g^2)
$$
holds where $o(\cdot)$ satisfies (\ref{exsecond5}).

%%%%%%%%%%%%%%%%%%%%%%%%%%%%%%%%%%%%%%%%%%%%%%%%%%%%%%%%%%%

\begin{thm}
The equality
$$
\p^2_- f(x_0) =\bigcap_{C_{\la} \in E^*_2(f)} C_\la,
$$
where $ C_{\la} = \p^2 h_{\la}^+(0)$, $ \la \in \La $, which form
the second-order upper Exhauster $E^*_2(f)$ of the function
$f(\cdot)$ at the point $x_0$. \label{thmexh3-1}
\end{thm}
{\bf Proof}. Let $h_{\la}^+(\cdot)$, $ \la \in \La,$ be the
exhauster $E^*_2(f)$ of the function $f(\cdot)$ at the point
$x_0$. For any $\la \in \La $ and $\al \in [0, \al_0]$, $g \in
B^n_1(0)$,
$$
f(x_0+\al g)-f(x_0) \leq h_{\la}^+(\al g) + o_{\la}(\al^2 g^2),
$$
where
$$
h_{\la}^+(g)=\max_{[ v,A] \in \p^2 h^+_{\la}(0)} [(v,g)+
\frac{1}{2}(Ag,g) ]
$$
is the  second-order upper approximation (UA) of the function
$f(\cdot)$ at $x_0$. That's why
$$
f(x_0+\al g)-f(x_0) \leq \inf_{\la \in \La} h_{\la}^+(\al g) +
o(\al^2 g^2),
$$
where $o(\al^2 g^2)= o_{\la_0}(\al^2 g^2) $, $\la_0 \in \La$.

Since $h_{\la}^+(\cdot)$, $ \la \in \La $, form  the second-order
Exhauster $E^*_2(f)$ of the function $f(\cdot)$ at the point
$x_0$, then
$$
\max_{[v,A] \in \bigcap C_{\la}} [\al (v,g)+\frac{\al^2}{2}
(Ag,g)] +o(\al^2 g^2) \leq f(x_0+\al g)-f(x_0) \,\,\,\, \forall
\al \in [0,\al_0], \, g \in B^n_1(0),
$$
where
$$
C_{\la}=\{ [w,B] \vl [ w,B] \in \p^2 h^+_{\la} (0) \}.
$$
From here \be \bigcap_{C_{\la} \in E^*_2 (f)} C_{\la} \subset
\p^2_- f(x_0). \label{exsecond6} \ee

On the other hand, for any $g \in B^n_1(0) $, $\al \in [0,\al_0],
$
$$
\max_{[v,A] \in \p^2_- f(x_0)} [\al(v,g)+\frac{\al^2}{2} (Ag,g)]
\leq f(x_0+\al g)-f(x_0) \leq h^+_{\la} (\al g) +o(\al^2 g^2)
\,\,\,\, \forall \la \in \La.
$$
Therefore, \be \p^2_- f(x_0) \subset \bigcap_{C_{\la} \in E^*_2
(f)} C_{\la}. \label{exsecond7} \ee

The statement of the theorem follows from (\ref{exsecond6}) and
(\ref{exsecond7}). $\Box$

%%%%%%%%%%%%%%%%%%%%%%%%%%%%%%%%%%%%%%%%%%%%%%%%%%%%%%%%%%%
%%%% Определение нижнего экзостера втрого порядка

\section{The upper Frechet subdifferential of the second order.
The lower second-order Exhauster}

We can introduce an upper second-order Frechet subdifferential
$$
\p_+^2 f(x_0)=\{ [w,B] \vl \limsup_{\| g \| \rar 0} \frac{1}{\| g
\|^2} (f(x_0+ g)-f(x_0)-(w,g)-\frac{1}{2}(B g,g)) \leq 0 \}.
$$
From the definition, it is clear that $ \p_+^2 f(x_0) $ is a
convex closed set.

We define the function $h^-_{\la}(\cdot): \mathbb{R}^n \rar
\mathbb{R}$
$$
h_{\la}^-(g)=\min_{[ v,A] \in \p^2 h^-_{\la}(0)} [(v,g)+
\frac{1}{2}(Ag,g) ],
$$
where $\p^2 h^-_{\la}(0) \subset \mathbb{R}^n \times
\mathbb{R}^{n^2}$ is a convex compact set, called the second
superdifferential of the function $h_{\la}^-(\cdot)$ at zero.
Here, as above, $(v,g)$ is the scalar product of vectors $v$ and
$g$.

We say that $h_{\la}^-(\cdot)$ is a second-order lower
approximation of the function $f(\cdot)$ at the point $x_0$ if
there exists $\al_0 >0$ such that \be f(x_0+\al g) - f(x_0) \meq
h_{\la}^-(\al g) + o_{\la}(\al^2 g^2) \,\,\, \forall g \in
S^1_1(0), \,\,\, \forall \al \in [0, \al_0], \al_0>0,
\label{exsecond7-2} \ee where \be \lim_{\al \rar 0^+}
\frac{o_{\la}(\al^2 g^2)}{\al^2} =0 \label{exsecond7-3} \ee
uniformly in $g \in B^{n}_1(0)$, \, $g^2 =\| g\|^2$.

We say that a set of the functions $h_{\la}^-(\cdot)$, $\la \in
\La$, forms the second-order lower exhauster $E_{*2}(f)$ of the
function $f(\cdot)$ at the point $x_0$ if inequality
(\ref{exsecond7-2}) is satisfied for any $\la \in \La$, where
$\La$ is a finite or infinite set of indices $\la$, and also
$$
f(x_0+\al g)- f(x_0) = \sup_{\la \in \La} h^-_{\la} (\al g) + o
(\al^2 g^2) \,\,\,\, \forall g \in B^n_1(0),
$$ where $ o(\cdot) $ satisfies conditions similar to those
(\ref{exsecond7-3}).

Note that if the set of functions $h^-_{\la}(\cdot)$, $\la \in
\La$, forms the second-order lower exhauster, then the set of
functions $h^-_{1\la}(\cdot):\mathbb{R}^n \longrightarrow
\mathbb{R}$, $\la \in \La$,
$$
h^-_{1\la}(g)=\min_{v \in \p h^-_{1\la}(0)} (v,g),
$$
where $ \p h^-_{1\la}(0)=\{ v \vl \exists A: (v,A) \in \p^2
h^-_{\la}(0) \}$, forms the first-order lower exhauster
$E_{1*}(h_1)$ of the function $h_1(g)=f'(x_0,g)$ at the point $0$.

Our goal will be to establish a connection between the second
superdifferentials $\p^2 h^-_{\la}(0)$ of the functions
$h_{\la}^-(\cdot)$, $\la \in \La$, and the second-order upper
Frechet subdifferential $\p^2_+ f(x_0)$.

Similarly to Theorem \ref{thmexh3-1}, we can prove the following
theorem.

\begin{thm}
The equality
$$
\p^2_+ f(x_0) =\bigcap_{C_{\la} \in E_{*2}(f)} C_\la
$$
is true, where $ C_{\la} = \p^2 h_{\la}^-(0)$, $ \la \in \La $ is
a finite or infinite set of indices of the functions $h_{\la}^-
(g)$ that form the second-order lower exhauster $E_{*2}(f)$ of the
function $f(\cdot)$ at the point $x_0$. \label{thmexh3-2}
\end{thm}

By the definition of the second-order lower exhauster, for any $Q
\subset E_{*2}(f)$, the following relations hold:
$$
f(x+\Delta x) \meq f(x_0) + \min_{[v,A] \in Q} [ (v,\Delta
x)+\frac{1}{2}(A \Delta x, \Delta x) ] + o(\Delta x^2)
$$
and
$$
f(x+\Delta x) = f(x_0) + \sup_{Q \in E_{*2}(f)} \min_{[v,A] \in Q}
[ (v,\Delta x)+\frac{1}{2}(A \Delta x, \Delta x) ] + o(\Delta
x^2).
$$
%%%%%%%%%%%%%%%%%%%%%%%%%%%%%%%%%%%%%%%%%%%%%%%%%%%%%%%%%%%%%%%%%%%%%%
\begin{rem}
Generally speaking, the upper (lower) Frechet subdifferentials of
the first (second) orders are introduced for arbitrary functions
that are not necessarily differentiable in directions
\cite{ioffepenot}.
\end{rem}

%%%%%%%%%%%%%%%%%%%%%%%%%Конец опредления нижнего экзостера 2-ого порядка

\vspace{0.5cm}

\section{Theorem on existence of the
second-order upper Exhauster}

\vspace{0.5cm}

Consider the functions $h_1(\cdot):\mathbb{R}^n \rar \mathbb{R}$,
$h_1(g)=f'(x_0,g)$ and $h_2(\cdot):\mathbb{R}^n \rar \mathbb{R}$,
$h_2(g)=f''(x_0,g,g) \overset{\mathrm{def}}{=} \frac{\p^2
f(x_0)}{\p g^2}$. The function $h_1(\cdot)$ is p.h. the first
order, and the function $h_2(g)$ is p.h. of the second order.

Our problem is to construct the first-order upper exhauster
$E^*_1(h_1)$ of the function $h_1(\cdot)$ and the second-order
upper exhauster $E^*_2(h_2)$ of the function $h_2(\cdot)$ at the
point zero.

It is known \cite{demrub} that if $h_1(\cdot)$ is u.s.c. and is
bounded above on $B_1^n(0)$, that is,
$$
\sup_{\| g \| \leq 1} h_1(g) < \infty,
$$
then there exists the first-order upper Exhauster $E^*_1(h_1)$ at
the point zero. Obviously, $E^*_1(f)$ at the point $x_0$ coincides
with $E^*_1(h_1)$. The exhaustive set of UCA. of the first order
consists of some p.h. convex functions $h_{1 \la}(\cdot)$, $\la
\in \La$, $\La $ is  a set of indices is finite or infinite. By
the definition of the exhaustive set of upper convex
approximations, we have
$$
f(x_0+\al g) = f(x_0)+ \inf_{\la \in \La} h_{1 \la}(\al g) + o(\al
g),
$$
where $ o(\al g) / \al \rar 0$ as $\al \rar 0^+$ uniformly in $g
\in B_1^n(0)$.

We will show that if the functions $h_1(\cdot)$ and $h_2(\cdot)$
are u.s.c. and bounded above on $B_1^n(0)$, that is,
$$
\sup_{\| g \| \leq 1} h_i(g) < \infty, \,\,\,\, i=1,2,
$$
then there exists an exhaustive set of upper approximations (UA)
of the second order, called the second order upper Exhauster
$E^*_2(f)$, and consisting of functions $h^+_{\la}(\cdot)$, $\la
\in \La$, for which (\ref{exsecond5a}) holds.

To prove that there exists a second-order upper  Exhauster for the
function $h_2 (\cdot)$, we use Lemma 5.2 on pp. 171-172
\cite{demrub}.

Consider the case when  a p.h. of the second-order function
$h_2(\cdot)$ is twice continuously differentiable on
$S^{n-1}_1(0)$. In this case, we obtain the form of the function
$h_2(\cdot)$ as the infimum of the maxima of some quadratic
functions.
\begin{thm}
For a p.h. of the second-order function $h_2(\cdot)$, twice
continuously differentiable on $S_1^{n-1}(0)$, there exists the
upper Exhauster $E^*_2(h_2)$ of the second order  at zero,
consisting of p.h. of the second-order functions $h_{2
\mu}(\cdot)$ of the form
$$
h_{2 \mu}(g) = \max_{A \in {\cal A_{\mu}}} (Ag,g) \,\,\, \forall g
\in B^n_1(0),
$$
where $\cal A_{\mu}$ is a convex compact set of matrices, such
that \be h_2(g)=\inf_{\mu \in \cal M} h_{2 \mu}(g) = \inf_{\mu \in
\cal M} \, \max_{A \in {\cal A_{\mu}}} (Ag,g), \label{exsecond8}
\ee where $ \cal M$ is a finite or infinite set of indices.
\label{thmexh4}
\end{thm}
{\bf Proof}. Take an arbitrary vector $g \in S^{n-1}_1(0)$ and
construct
$$
\hat f(x_0 + \al g)=f(x_0 + \al g) - f(x_0) - \al h_1(g)
=\frac{\al^2}{2} h_2(g)+o(\al^2 g),
$$
where $h_2(\cdot)$ is defined above. The function $h_2(\cdot)$ is
p.h. of the second-order, i.e., $h_2(t g)=t^2 h_2(g)$ for any
$t>0$. Clearly, that UA of the functions $h_2(\cdot)$ on $B_1^n
(0)$ must be sought among p.h. functions of the second-order,
since $h_2(\cdot)$ is one of them.

The simplest  p.h. function of the second order is a function of
the form $(Ag,g)$, where $A=A[n \times n]$ is a matrix of
dimension $n \times n$, $g \in \mathbb{R}^n$. Let $A=A(g)=\{
a_{ij}(g) \vl i,j \in 1:n \}$. Consider the function
$\Theta(\cdot): \mathbb{R}^{n\times n} \rar \mathbb{R}$
$$
\Theta(u)=\sum_{i,j=1}^n a_{i,j}(g) u_{ij},
$$
which is linear for fixed $g$ in $u \in \mathbb{R}^{n \times n}$,
$u=\{ u_{11}, u_{12}, \dots u_{nn} \}$. The vector $g$ is
expressed in terms of the coordinates of the vector $u=\{ g_i g_j
\vl i,j \in 1:n \}$. Therefore, $\Theta(\cdot)$ is considered as a
function of the vector $u \in \mathbb{R}^{n \times n}$.

If we set $u_{ij}=g_i g_j$, then
$$
\Theta(u)=\sum_{i,j=1}^n a_{ij}(g)g_i g_j =(A(g)g,g).
$$
If $A(g)=h_2''(g)$, $g \in S^{n-1}_1(0)$, then for $u_{ij}=g_i
g_j$ we will have
$$
\Theta(u)=f''(x_0,g,g).
$$
According to the conditions of the theorem, $\Theta(\cdot)$ is
u.s.c. with respect to the variable $ u=\{u_{ij} \vl i,j \in 1:n
\}$ where $u_{ij}=g_i g_j $. The function $\Theta(\cdot)$ can be
extended to the entire sphere $S^{n^2-1}_1(0)$, preserving upper
semicontinuity. Therefore, we assume that $\Theta(\cdot)$ is
u.s.c. at any point of the sphere $S_1^{n^2-1}(0)$.

According to Rubinov's theorem, \cite{demrub} there exists an
exhaustive set $E^*_1(\Theta)$ of the function $\Theta(\cdot)$ at
zero. The set $E^*_1(\Theta)$ consists of sublinear functions
$\pi_{\mu}(\cdot): \mathbb{R}^n \times \mathbb{R}^n \rar
\mathbb{R}$, for which

\be \Theta(u)=\inf_{\mu \in \cal M} \pi_{\mu}(u),
\label{exsecond9} \ee where $\cal M$ is a some set of indices.

If in (\ref{exsecond9}) we set $u_{ij}=g_i g_j$, we obtain the
equality
$$
f''(x_0,g,g)=\inf_{\mu \in \cal M} \left.\pi_{\mu}(u)\right\vert_{
u=\{g_i g_j \vl i,j \in 1:n \}}.
$$
Here $\pi_{\mu}(\cdot)$ are sublinear convex functions.  The
representation
$$
\pi_{\mu}(u)=\max_{\nu \in N(\mu)} (a_{\nu}, u), \,\,\, a_{\nu}
\in \mathbb{R}^{n^2}
$$
is correct for any $\mu \in \cal M $, the set $N(\mu)$ is a finite
or infinite set of indices $\nu $ for each $\mu \in \cal M$. Then,
for $u=\{g_i g_j \vl i,j \in 1:n \} $
$$
f''(x_0,g,g)=\inf_{\mu \in \cal M} \max_{\nu \in N(\mu)}
\left.(a_{\nu}, u) \right\vert_{ u=\{g_i g_j \vl i,j \in 1:n \}}.
$$
But for $ u= \{g_i g_j \vl i,j \in 1:n \}$
$$
\max_{\nu \in N(\mu)} (a_{\nu}, u) = \max_{A \in {\mbox{co}}\{
A_{\nu} \vl \nu \in N(\mu)\}} (Ag,g),
$$
where the matrices $A_{\nu}$ are composed of the vectors
$a_{\nu}\in \mathbb{R}^{n^2}$ in such a way that the equality
$$
(A_{\nu}g,g)=\sum_{i,j = 1}^n a_{\nu}^{ij} g_i g_j,
$$
holds if  $a_{\nu}=\{ a_{\nu}^{ij} \vl i,j \in 1:n \}$.

As a result, we obtain that
$$
f''(x_0,g,g)=\inf_{\mu \in \cal M} \max_{\nu \in N(\mu)}
(A_{\nu}g,g)= \inf_{\mu \in M} \max_{A \in {\mbox{co}}\{ A_{\nu}
\vl \nu \in N(\mu)\}} (Ag,g),
$$
that is, the function $h_2(\cdot)$ is the infimum of functions
that are sublinear with respect to $A_{\nu}$ and p.h. of the
second order with respect to $g$. Therefore, the existence of the
 Exhauster of the second-order for the function $h_2(\cdot)$ at zero is
proven. The theorem is proved. $\Box$
\begin{rem}
If $h_2(\cdot)$ is twice continuously differentiable on
$S^{n-1}_1(0)$ and p.h. of the second order, then there is the
lower exhauster $E_{*2}(h_2)$ of the second-order of $h_2(\cdot)$
at zero. It can be easily proved by replacing $h_2(\cdot)$ with
$-h_2(\cdot)$.
\end{rem}

Let $h_2(\cdot)$ be a bounded above and u.s.c. on $B^n_1(0)$, p.h.
of the second-order function. By the definition of the upper
semicontinuity  of the function $h_2(\cdot)$ at a point $g_0 \in
S^{n-1}_1(0)$, for an arbitrary $\eps>0$ there exists a $\del$
neighborhood $S_{\del}(g_0)$ of the point $g_0$  where the
inequality
$$
h_2(g) \leq h_2 (g_0) + \eps \,\,\,\, \forall g \in S_{\del}(g_0)
$$
holds.  Therefore, there exists an $\eps$ strip for $S_1^{n-1}(0)$
where the function $h_2(\cdot)$ can be approximated from above,
that means the inequality written above is satisfied,  by a twice
continuously differentiable p.h. of the second-order function, for
which there exists the upper second-order Exhauster at zero.
Therefore, the proven Theorem \ref{thmexh4} on the existence of
the upper second-order Exhauster $E^*_2(h_2)$ at zero will be true
for any u.s.c.  and bounded above on $S^{n-1}_1(0)$, p.h. of the
second-order function $h_2(\cdot)$. The boundedness of the
function $h_2(\cdot)$ on $S^{n-1}_1(0)$ guarantees that the upper
convex approximations will also be bounded on $S^{n-1}_1(0)$,
which follows from the proof of  Rubinov's theorem \cite{demrub},
pp. 171-172. Since we have reduced any p.h. of the second-order
function to a p.h. function of the first order, then from Theorem
\ref{thmexh0} we obtain
\begin{thm}
For any bounded, u.s.c. on $B^n_1(0)$, p.h. of the second order
function $h_2(\cdot)$, there exist the upper  $E^*_2(h_2)$ and the
lower $E_{*2}(h_2)$ Exhausters of the second-order at zero.

For any bounded, l.s.c. on $B^n_1(0)$, p.h. of the second order
function $h_2(\cdot)$, there exist the upper  $E^*_2(h_2)$ and the
lower $E_{*2}(h_2)$ Exhausters of the second-order at zero.
\label{thmexh5}
\end{thm}

Under the assumptions about boundedness from above and u.s.c. of
the functions $h_1(\cdot)$, $h_2(\cdot)$, there exists the upper
second-order Exhauster $E^*_2(f)$ of the function $f(\cdot)$ at
the point $x_0$, consisting of the functions
$$
h^+_{\tau}(g)=h_{1\la}(g)+\frac{1}{2}h_{2\mu}(g),
$$
where $\tau=[\la,\mu]$ is a set of parameters, that is,
$$
f(x+g) \leq f(x_0)+h^+_{\tau}(g)+o(g^2) \,\,\, \forall g \in
B^n_1(0), \, \tau \in T=\La \times M,
$$
and \be f(x+g)=f(x_0)+\inf_{\tau \in T} h^+_{\tau}(g)+o(g^2).
\label{exsecond10} \ee

We have

\begin{conseq}
Under the conditions of Theorems  \ref{thmexh0} and \ref{thmexh5},
there exist the upper Exhausters $E^*_1(h_1)$ and $E^*_2(h_2)$ of
the first and second orders of the functions $h_1(\cdot)$ and
$h_2(\cdot)$ at zero respectively. Therefore, there exist the
upper Exhausters $E^*_1(f)$ and $E^*_2(f)$ of the first and second
orders of the function $f(\cdot)$ at the point $x_0$.
\end{conseq}

The question arises: how to construct $E^*_1(f)$ and $E^*_2(f)$?

The article \cite{proudexhaust} shows how to construct $E_1^*(f)$
for any Lipschitz function $f(\cdot): \mathbb{R}^n \rar
\mathbb{R}$ with Lipschitz constant $L$. We will use the results
of this article to construct the second-order Exhauster of the
function $f(\cdot)$ at the point $x_0$.

Note that for a Lipschitz directionally differentiable function
$f(\cdot)$, the function $h_1(g)=f'(x_0,g)$ is a continuous
function in $g$. Indeed,
$$
\lim_{u \rar g} | h(u)-h(g) | = \lim_{u \rar g}
\left\vert\lim_{\al \rar 0^+} \frac{f(x_0+\al u)-f(x_0)}{\al} -
\lim_{\al \rar 0^+} \frac{ f(x_0+\al g)-f(x_0)}{\al} \right \vert=
$$
$$
=\lim_{u \rar g} \lim_{\al \rar 0^+} \left\vert \frac{f(x_0+\al
u)-f(x_0)}{\al} - \frac{ f(x_0+\al g)-f(x_0)}{\al} \right \vert =
$$
$$
=\lim_{u \rar g} \lim_{\al \rar 0^+} \left\vert \frac{ f(x_0+\al
u)-f(x_0+\al g)}{\al} \right \vert \leq \lim_{u \rar g} \lim_{\al
\rar 0^+} L \| u -g \|=0,
$$
that is, the continuity of $h_1(\cdot)$ is proven.

The algorithm for constructing the upper Exhauster $E^*_1(f)$ of
the first order of the function $f(\cdot)$ at the point $x_0$ was
described above. It is clear that $E^*_1(f)=E^*_1(h_1)$. Next, we
proceed to constructing the second-order upper exhauster
$E^*_2(h_2)$.

We will assume that $h_2(\cdot)$ is a continuous p.h. of the
second-order function twice differentiable  almost everywhere
(a.e.) in $B^n_1(0)$. The set of points where $h_2''(\cdot)$
exists is denoted by ${\cal N}_2(h_2)$. For $g \in B^n_1(0) \cap
{\cal N}_2(h_2)$, the function $h_2(\cdot)$ has the form
$h_2(g)=(A(g) g, g)$. Clearly,  the equality $h''_2(g) = A(g)$
holds for such $g$.

The last equality, as mentioned earlier, can be represented in the
form of the p.h. function of the first order in the extended space
$$
\Theta(u)=(a(u), u) \,\,\,\, u \in \mathbb{R}^{n^2}.
$$
If we set $u=(g_1g_1, g_1g_2, \dots, g_n g_n) \in
\mathbb{R}^{n^2}$ and $a(u)=(a_{11},a_{12}, \dots, a_{nn}) \in
\mathbb{R}^{n^2}$, $a_{ij}(u)$ are the elements of the matrix
$A(g)$, then
$$
\Theta(u)=(A(g)g,g)=h_2(g)
$$
and
$$
A(g)=h''_2(g).
$$
We will assume that
$$
\max_{g \in B^n_1(0) \cap {\cal N}_2(h_2)} \| h''_2(g) \| \leq M.
$$
For $u=(g_1g_1, g_1g_2, \dots, g_n g_n)$ and $g \in B^n_1(0) \cap
{\cal N}_2(h_2)$
$$
\Theta(u)=h_2(g)
$$
and \be \| \Theta'(u)\| \leq M. \label{exsecond11a} \ee

The function $\Theta(\cdot)$ can be extended to the entire ball
$B^{n^2}_1(0)$ in such a way that inequality (\ref{exsecond11a})
held  for all $u \in B^{n^2}_1(0) \backslash \{0\}. $ The
Exhauster constructed for the function $\Theta(\cdot)$ at zero
will be the second-order Exhauster for the function $h_2(\cdot)$
for the corresponding $g$ and $u$, as discussed above.

Let
$$
\hat f(x)=f(x)-f(x_0)-h_1(x-x_0).
$$
The construction of the  Exhauster of the second-order will
proceed similarly to the construction of the  Exhauster of the
first-order for a p.h. function of the first order using the
cutoff method \cite{proudexhaust}.

We define
$$
\bar f(x)=\hat f(x)+\frac{1}{2}(M(x-x_0), x-x_0).
$$
Since the second directional derivative can be represented as a
scalar product, we can repeat the algorithm for constructing the
upper Exhauster for p.h. function of the first-order.

%%%%%%%%%%%%%%%%%%%%%%%%%%%%%%%%%%%%%%%%%%%%%%%%%%%%%%%%%%%%%%%

\section{Construction of the upper Exhauster of the second-order}

Assuming continuity of $h_2(\cdot)$ and a.e. differentiability on
$B^n_1(0)$, we give a rule for constructing the  upper Exhauster
of the second-order of the function $h_2(\cdot)$ at zero. As
before, we denote the set of differentiability points of
$h_2(\cdot)$ on $B^n_1(0)$ by ${\cal N}_2(h_2)$.

For an arbitrary vector $g\in S^{n-1}_1(0)$ and a matrix $A[n
\times n]$, we define the sets
$$
P_g(A)=\{ B[n \times n] \vl ( Bg,g) \leq (Ag,g) \}.
$$

For $g \in {\cal N}_2(h_2) \bigcap S^{n-1}_1(0) $, we calculate
$A(g)= h_2''(g)$ and define the set of matrices
$$
Q_2(\bar f)=\overline{\bco}\{B[n \times n] \vl B=\bar h_2''(g)
\,\,\, \forall g \in {\cal N}_2(h_2) \bigcap S^{n-1}_1(0) \},
$$
where $\bar h_2(g)=\bar f''(x_0,g,g)$. Obviously, $ \bar
h_2(g)=h_2(g)+M$ for $g \in S_1^{n-1}(0)$. The exhaustive set of
the  upper approximations of the second-order $E^*_2(\bar f)$ for
the function $\bar f(\cdot)$ at the point $x_0$ consists of
functions of the form $\frac{1}{2}(Cg,g)$. We will determine
$E^*_2(\bar f)$ if we know the set of matrices $C[ n \times n ]$
from which quadratic functions are constructed. The set
$E^*_2(\bar f)$ consists of matrices $C_{2g}[n \times n] $: \be
E^*_2(\bar f)=\{ C_{2g}[n \times n] \vl C_{2g}=P_g(A(g)+M \, I[n
\times n ] ) \bigcap Q_2(\bar f) \,\,\, \forall g \in {\cal
N}_2(h_2) \bigcap S^{n-1}_1(0) \}. \label{exsecond11} \ee Here
$A(g)=h''_2(g)$ for $g \in {\cal N}_2(h_2) \bigcap S^{n-1}_1(0)$,
$I[n \times n]$ is the identity matrix of dimension $n \times n$
with one on the main diagonal and the remaining terms equal to
zero.

In the article \cite{proudexhaust}, it was proved that the method
of cutting hyperplanes (planes) yields an exhaustive set of convex
compact sets, from which we construct convex p.h. of the
first-order functions that form the Exhauster of the first order.
Since $(A(g)g,g)$, where $A=h''_2(g)$, can be represented as a
scalar product in the extended space, all the arguments in the
article \cite{proudexhaust} for constructing the Exhauster of the
first-order are carried over to constructing the Exhauster of the
second-order. Therefore, for symmetric matrices $W$: $W^T = W$,
$W^T$ is a transposed matrix,
$$
E^*_2(\hat f)=\{ W[n \times n]+C_{2g}[n \times n] \vl \| W[n
\times n] \| \leq M, \, C_{2g}[n \times n] \in E^*_2(\bar f),
W^T=W \}
$$
consists of the sets from which one can construct p.h. functions
of the second-order  that are the second-order upper
approximations (UA) of the function $\hat f(\cdot)$ at $x_0$.

The  Exhauster of the second-order of the function $f(\cdot)$ at
$x_0$ is equal to
$$
E^*_2(f)=\{ \big[ w+C_{1g}, W[n \times n]+C_{2g}[n \times n] \big]
\vl w \in L B^n_1(0), \, \| W[n \times n] \| \leq M, W^T = W,
$$
$$
C_{1g} \in E^*_1(\tilde f), C_{2g}[n \times n] \in E^*_2(\bar f),
\forall g \in {\cal N}_2(h_2) \bigcap S^{n-1}_1(0) \}.
$$
Here, as before,
$$
\tilde f(x)=f(x)+L \|x - x_0 \|.
$$
According to the definition of the second-order Exhauster, for any
$Q \subset E^*_2(f)$ the following relations hold:
$$
f(x+\Delta x) \leq f(x_0) + \max_{[v,A] \in Q} [ (v,\Delta
x)+\frac{1}{2}(A \Delta x, \Delta x) ] + o(\Delta x^2),
$$
where $o(\Delta x^2)=o(\| \Delta x \|^2)$, $o(\Delta x^2) / \|
\Delta x \|^2 \longrightarrow 0$ for $\Delta x \rar 0$, and
$$
f(x+\Delta x) = f(x_0) + \inf_{Q \in E^*_2(f)} \max_{[v,A] \in Q}
[ (v,\Delta x)+\frac{1}{2}(A \Delta x, \Delta x) ] + o(\Delta
x^2).
$$

%%%%%%%%%%%%%%%%%%%%%%%%%%%%%%%%%%%%%%%%%%%%%%%%%%%%%%%%%%%%%%

\section{Construction of the lower Exhauster of the second order }

Assuming continuity of $h_2(\cdot)$ and a.e. differentiability on
$B^n_1(0)$, we give a rule for constructing   lower Exhauster of
the second-order of the function $h_2(\cdot)$ at zero. The set of
differentiability points of $h_2(\cdot)$ on $B^n_1(0)$ is denoted,
as before, by ${\cal N}_2(h_2)$.

Similarly to Theorem 3 \cite{proudexhaust}, we can construct a
 lower Exhauster of the second-order for the function $f(\cdot)$ at the
point $x_0$. We define the function
$$
\underline f(x)=\hat f(x)-\frac{1}{2}(M(x-x_0), x-x_0)
$$
and the set
$$
Q_2(\underline f)=\overline{\bco}\{B[n \times n] \vl B=\underline
h_2''(g) \,\,\, \forall g \in {\cal N}_2(h_2) \bigcap S^{n-1}_1(0)
\},
$$
where
$$
\underline h_2(g)=\underline f''(x_0, g,g)=\frac{\p^2 \underline
f(x_0)}{\p g^2}.
$$
Just as in constructing the upper Exhausters, we will search for a
set of the  matrices $C[n \times n]$ from which we can construct
the  p.h of the second order functions of the form $(Cg,g)$ that
form the exhaustive set of the second-order lower approximations
of the function $\underline f(\cdot)$.

The exhaustive set $E_{*2}(f)$ of the second-order lower
approximations of the function $f(\cdot)$ at the point $x_0$ is
equal to
$$
E_{*2}( f)= \{ \big[ w+C_{1g}, W[n \times n]+C_{2g}[n \times n]
\big] \vl w \in L B^n_1(0), \, \| W[n \times n] \| \leq M, W^T =
W,
$$
$$
C_{1g} \in -E^*_1(-\underset{*}{f}), C_{2g}[n \times n] \in
-E^*_2(-\underline f), \forall g \in {\cal N}_2(h_2) \bigcap
S^{n-1}_1(0) \}.
$$
Here
$$
\underset{*}{f}(x)=f(x)- L \|x - x_0 \|.
$$
Here $E^*_2(-\underline f)$ is defined by the formula
(\ref{exsecond11}) with the function $\bar f(\cdot)$ replaced by
$-\underline f(\cdot)$.

By the definition of the second-order lower Exhauster, the
following relations hold:  for any $Q \subset E_{*2}(f)$
$$
f(x+\Delta x) \meq f(x_0) + \min_{[v,A] \in Q} [ (v,\Delta
x)+\frac{1}{2}(A \Delta x, \Delta x) ] + o(\Delta x^2),
$$
and
$$
f(x+\Delta x) = f(x_0) + \sup_{Q \in E_{*2}(f)} \min_{[v,A] \in Q}
[ (v,\Delta x)+\frac{1}{2}(A \Delta x, \Delta x) ] + o(\Delta
x^2).
$$
%%%%%%%%%%%%%%%%%%%%%%%%%%%%%%%%%%%%%%%%%%%%%%%%%%%%%%%%%%%%

%%%%%%%%%%%%%%%%%%%%%%%%%%%%%%%%%%%%%%%%%%%%%%%%%%%%%%%%%%%%%%%%%%

\section{Second-order subdifferential of the function $h_2(\cdot)$}

The  subdifferential  of the second-order $\Psi^2 f (\cdot)$ was
introduced for any Lipschitz function $f(\cdot):\mathbb{R}^n \rar
\mathbb{R}$ in \cite{pimfirstsecondsubdiff}. Let us construct the
second-order subdifferential $\Psi^2 h_2 (0)$ at zero for the p.h.
second-order function $h_2(\cdot):\mathbb{R}^n \rar \mathbb{R}$,
$h_2(g)=f''(x_0,g,g)$. We will show how it is related to the set
$$
Q_2(f)=\overline{\bco}\{B[n \times n] \vl B= h_2''(g) \,\,\,
\forall g \in {\cal N}_2(h_2) \bigcap S^{n-1}_1(0) \}.
$$
In \cite{pimTwiceContinCodif} the case was considered when
$$
h (q) = \max_{A \in \cal A} \, \frac{1}{2} (Aq, q).
$$
It was proved that the second-order subdifferential of the
function $h(\cdot)$ at zero is equal to
$$
\Psi^2 h (0)= {\cal A}.
$$
Here we will consider a more general case and show which matrices
make up the second-order subdifferential $\Psi^2 h_2(0)$ of the
function $h_2(\cdot)$ at zero.

We will assume that $h_2(g)=f''(x_0, g,g)$ is a continuous p.h.
function  of the second-order  in $g$, a.e. twice differentiable
and
$$
\| h''_2(g) \| \leq M \,\,\,\, \forall g \in S^{n-1}_1(0).
$$

Let for any $g \in S_1^{n-1}(0)$ the expansion holds
$$
f(x_0+\al g)=f(x_0)+\al f'(x_0,g)+\frac{\al^2}{2}
f''(x_0,g,g)+o(\al^2 g^2),
$$
where $\frac{o(\al^2 g^2)}{\al^2} \longrightarrow $ as $\al \rar
0^+$ uniformly in $g \in B^n_1(0)$.

\begin{thm}
The equality is true
$$
\Psi^2 h_2(0) = \Psi^2 \hat f(x_0) = Q_2(f),
$$
where $\Psi^2 \hat f(x_0)$ is the second-order subdifferential of
the function $\hat f(x)=f(x)-f(x_0)-h_1(x-x_0)$ at the point
$x_0$, $h_1(g)=f'(x_0,g)$, $\Psi^2 h_2 (0)$ is the second-order
subdifferential of the function $h_2(\cdot)$ at zero.
\end{thm}
{\bf Proof}. It is clear that
$$
f''(x_0,g,g)=\hat f''(x_0,g,g),
$$
i.e.
$$
h''_2(g)=f''(x_0,g,g)= \hat f''(x_0,g,g) = \hat h''_2(g) \,\,\,\,
\forall g \in {\cal N}_2(h_2),
$$
and therefore,
$$
Q_2(f)=Q_2(\hat f) =\overline{\bco}\{B[n \times n] \vl B= \hat
h_2''(g) \,\,\, \forall g \in {\cal N}_2(h_2) \}.
$$
If we show that
$$
\Psi^2 \hat h_2(0)= \Psi^2 \hat f (x_0)= Q_2(\hat f)
$$
and
$$
\Psi^2 \hat h_2 (0) = \Psi^2 h_2 (0),
$$ then the theorem will be proved.

For $g \in {\cal N}_2(h_2)$ \be f(x_0+\al g)=f(x_0)+\al
(a(g),g)+\frac{\al^2}{2}(A(g)g,g)+o(\al^2 g^2), \label{exsecond12}
\ee where
$$
A(g)=h''_2(g), \,\, a(g)=h_1'(g), \,\,h_1(g)=f'(x_0,g).
$$
In \cite{pimfirstsecondsubdiff}, a class $\Xi$ of multivalued
mappings (MV) $D(\cdot):\mathbb{R}^n \rar 2^{\mathbb{R}^n}$ with
convex compact images is introduced. These maps are used to
construct averaged integrals and introduce the subdifferentials of
the first and second orders. All $D(\cdot)$ contain $x_0$, and the
diameter $d(D(x))$ tends to zero as $x \rar x_0$. For each
$D(\cdot)$, the averaging functions $\varphi_D(\cdot)$ and
$\psi_D(\cdot)$ are constructed, as well as their subdifferentials
at the point $x_0$ in the form of  limits of the gradients or
matrices of the second mixed derivatives for $x \rar x_0$, which
are calculated in the regions of constancy of $D(\cdot)$. The
subdifferential of the first $\Phi f(x_0)$ and second $\Psi^2
f(x_0)$ orders is the union over all $D(\cdot) \in \Xi$ of the
limit gradients and matrices of the second mixed derivatives of
the functions $\varphi_D(\cdot)$ and $\psi_D(\cdot)$ for $x \rar
x_0$.

%%%%%%%%%%%%%%%%%%%%%%%%%%%%%%%%%%%%%%%%%%%%%%%%%%%%%%%%%%%%%%%%%%

We define the functions \cite{pimfirstsecondsubdiff} $ \varphi_{D}
(\cdot): \mathbb {R} ^ n \rar \mathbb {R} $ and $ \psi_{D}
(\cdot): \mathbb {R} ^ n \rar \mathbb {R} $ for constructing the
subdifferentials \be \varphi_{D} (x): = \frac {1} {\mu (D (x))}
\int_ {D (x)} f (x + y) d y, \label{exsecond13} \ee \be \psi_{D}
(x): = \frac {1} {\mu (D (x))} \int_ {D (x)} \varphi_D (x + y) d
y, \label{exsecond14} \ee where $\mu (D (x)) > 0$ is the Lebesgue
measure, $ D (\cdot): \mathbb {R} ^ n \rightarrow 2^{\mathbb {R} ^
n} $ is a set-valued mapping (SVM) from  $\Xi $
\cite{pimfirstsecondsubdiff}.

The function $\varphi_{D} (\cdot)$ is used to construct the
subdifferential of the first-order, and the function $\psi_{D}
(\cdot)$ is used to construct the  subdifferential of the
second-order.

We define a set of matrices \be \p ^ 2 \psi_D (x_0) = \mbox{co}
\{A \in \mathbb {R} ^ {n \times n} \vl A = \lim_ {x_i \rar x_0}
\nabla^2 \psi_ {D} (x_i) \}, \label{exsecond15} \ee where the
points $x_i $ belong to the regions of constancy of the SVM $ D
(\cdot) \in \Xi $, $\nabla^2 \psi_ {D} (\cdot)$ is a matrix of the
second mixed derivatives $ \psi_ {D} (\cdot) $.

We define a SVM $ \Psi^2 f (\cdot): \mathbb {R}^ n \rar 2 ^
{\mathbb {R} ^ {n \times n}} $ with images \be \Psi ^ 2 f (x_0) =
\mbox{co} \, \bigcup_ {D (\cdot)} \, \p ^ 2 \psi_D (x_0),
\label{exsecond16} \ee where the union is taken over all SVMs $ D
(\cdot) \in \Xi $. We call the set $ \Psi^2 f (x_0) $  the {\em
second-order subdifferential} of the function $ f (\cdot) $ at the
point $ x_0 $ \cite{pimfirstsecondsubdiff}.

When calculating $\psi_{D}(\cdot)$, $D(x_0+\Delta x)$ can be
divided into sectors $S_i(\Delta x)$ with a common vertex $x_0$.
In each sector $S_i(\Delta x)$ the function $h_2(\cdot)$ is
approximated by the function $(A(g_i)g,g)$, where
$A(g_i)=h_2''(g_i)$, $g_i \in {\cal N}_2(h_2)$, $g_i \in
S_i(\Delta x)$.

It was proved in the article \cite{pimfirstsecondsubdiff} that for
any twice-differentiable function $f(\cdot)$ at $x_0$ for any
$D(\cdot) \in \Xi$ \be \psi_D''(x)=f''(x_0)+o''(\|x-x_0\|^2),
\label{exsecond17} \ee where $o(\cdot)$ is a second-order
infinitesimal function with respect to $\| x-x_0\|$ in a
neighborhood of $x_0$. From this, from (\ref{exsecond15}) and
(\ref{exsecond16}), we obtain the equality
$$
\Psi^2 f(x_0)=\{ f''(x_0) \}.
$$
In our case, instead of the function $f(\cdot)$, which is not
twice continuously differentiable at the point $x_0$, we consider
the function $\hat f(\cdot)$. In each sector $S_i(\Delta x)$,
containing the vector $g_i$, the function $\hat f(\cdot)$ is
approximated by the function
$$
(A(g_i) g, g)+o( \Delta x ^2).
$$
The function $\hat \psi_D(\cdot)$, calculated for $D(x+\Delta
x)=S_i(\Delta x)$, is equal, according to (\ref{exsecond17}), to
$$
\hat \psi''_D(x)=A(g_i)+o''(\| x-x_0 \|^2) \,\,\, \forall x \in
S_i(\Delta x).
$$
Here $o(\| x-x_0 \|^2) $ is an uniformly infinitesimal function in
all sectors, since in the expansion (\ref{exsecond12}), by the
assumption, $o(\al^2 g^2)$ is a second-order uniformly
infinitesimal function  in $g \in B^n_1(0) \bigcap S_i(\Delta x)$.

When calculating $\hat \psi''_D (\cdot)$ for an arbitrary
$D(\cdot) \in \Xi$, the second derivative can be represented as
the convex hull of the matrices $A(g_i)$, $g_i \in S_i(\Delta x)$,
$g_i \in {\cal N}_2 (h_2)$, i.e.
$$
\hat \psi''_D(x) \simeq \sum_{i=1}^m \al_i A(g_i) +o'' ( \| x -
x_0 \|^2),
$$
where $0 \leq \al_i \leq 1$, $\sum_1^m \al_i=1$. For $x \rar x_0$
and $m \rar \infty$ we have
$$
\Psi^2 \hat f(x_0) =\overline{\bco} \{ A(g_i) \vl A(g_i)=\hat
h_2''(g_i), \,\, g_i \in {\cal N}_2(h_2)\} =
$$
$$
= \overline{\bco} \{ A(g_i) \vl A(g_i)= h_2''(g_i), \,\, g_i \in
{\cal N}_2(h_2)\}=Q_2(f)=\Psi^2 h_2 (0),
$$
since
$$
\hat h''_2(g_i)=h''_2(g_i)=\hat f''(x_0,g_i,g_i)=
f''(x_0,g_i,g_i),
$$
which is what was required. The theorem is proven. $\Box$

%%%%%%%%%%%%%%%%%%%%%%%%%%%%%%%%%%%%%%%%%%%%%%%%%%%%%%%%%%%%%%%%%%

\section{\bf Conclusion}

The paper introduces the upper (lower) second-order approximations
for an arbitrary bounded  u.s.c. or l.s.c. function. The concept
of the upper and lower Exhauster of the second order is
introduced. The theorems establishing the connection between the
upper (lower) second-order Frechet subdifferentials and the lower
(upper) second-order Exhausters, respectively, are proved. Based
on the results of the author \cite{proudexhaust} of constructing
the Exhauster of the first-order, the procedure for constructing
the Exhausters of the second-orders given.

The importance of constructing the second-order Exhausters is
explained by the fact that they enable us to write the necessary
and sufficient conditions for the second-order optimality.

The following statement is easy to prove. For the function
$f(\cdot)$ to have a minimum at $x_0$, it is necessary and
sufficient that for all $\la \in \La$ the functions
$h^+_{\la}(\cdot)$, forming the  upper Exhauster $E_2^*(f)$ of the
second-order, have a minimum at zero.

A necessary condition for the minimum of the function $f(\cdot)$
at $x_0$ is \be 0 \in \p h^+_{1\la}(0) \,\,\,\,\,\, \forall \la
\in \La. \label{exsecond18} \ee A sufficient condition for a
minimum is
$$
0 \in \mbox{int } \p h^+_{1\la}(0) \,\,\,\,\,\, \forall \la \in
\La.
$$
If the condition (\ref{exsecond18}) is satisfied and zero belongs
to the boundary of the set $\p h^+_{1\la_0}(0)$ for $\la_0 \in
\La$, that is, for some direction $g_0 \in B^n_1(0)$ the equality
$ h^+_{1\la_0}(g_0) =0 $ holds, then a sufficient condition for a
minimum in this case is $(Ag_0,g_0) > 0$ for all
$$
A \in \{ B \vl (0, B) \in \p^2 h^+_{\la_0}(0) \}.
$$
Note that the inequality written above must hold for all $\la_0$
and $g_0$ specified above. This can be reformulated as follows:
for a point $x_0$ to be a minimum point of the function
$f(\cdot)$, it is  sufficient that for all sufficiently small
$\alpha \in (0, \al_0),$ $\al_0>0$, the following inequality
$$
h^+_{\lambda}(\al g) >0 \,\,\, \forall \al \in (0, \al_0), \,\,
\forall g \in S^{n-1}_1(0), \,\,\, \forall \lambda \in \Lambda
$$
holds.

A necessary condition for the minimum of the function $f(\cdot)$
is that for all sufficiently small $\alpha \in (0, \al_0),$
$\al_0>0$, the inequality
$$
h^+_{\lambda}(\al g) \meq 0 \,\,\, \forall \al \in (0, \al_0),
\,\, \forall g \in S^{n-1}_1(0), \,\,\, \forall \lambda \in
\Lambda
$$
holds.

We can also write  sufficient condition for the maximum of the
function $f(\cdot) $ at the point $x_0$:
$$
h^-_{\lambda}(\al g) < 0 \,\,\, \forall \al \in (0, \al_0), \,\,
\forall g \in S^{n-1}_1(0), \,\,\, \forall \lambda \in \Lambda.
$$

A necessary condition for the maximum of the function $f(\cdot)$
is that for all sufficiently small $\alpha \in (0, \al_0),$
$\al_0>0$, the inequality
$$
h^-_{\lambda}(\al g) \leq 0 \,\,\, \forall \al \in (0, \al_0),
\,\, \forall g \in S^{n-1}_1(0), \,\,\, \forall \lambda \in
\Lambda
$$
holds.

Based on the necessary and sufficient conditions for the minimum
(maximum), we can further develop numerical methods for finding
extremum points.

%%%%%%%%%%%%%%%%%%%%%%%%%%%%%%%%%%%%%%%%%%%%%%%%%%%%%%%%%%%%%%


\newpage

\begin{thebibliography}{14}

\bibitem{andramonov} Andramonov M.Yu. Calculating Quasi-Differentials and
Exhausters using values of function // Computational Methods and
Programming. 2006, Vol. 7, No. 2, pp. 190-194.

\bibitem{clarke}  Clarke F. Generalized Gradients and Applications // Trans.
Amer. Math. Soc. 1975, Vol. 205, No. 2, pp. 247-262.

\bibitem{demrub} Demyanov V.F. Rubinov, A.M. Fundamentals of Nonsmooth Analysis
and Quasi-Differential Calculus. Moscow: Nauka, 1990, pp. 432-433.

\bibitem{demroshina} Demyanov V. F., Roshchina, V. A. Generalized Subdifferentials
and Exhausters // Vladikavkaz Mathematical Journal, 2006. Vol. 8,
Issue 4, pp. 19-31.

\bibitem{demyanovnexhauster} Demyanov V.F. Exhausters of a Positively Homogeneous
Function // Optimization, Vol. 45, 1999. pp. 13-29.

\bibitem{demyanovnexhauster2} Demyanov V.F.  Exhausters and
Convexificators. New Tools in Nonsmooth Analysis //
Quasi-Differentiability and Related Topics. Dordrecht: Kluwer
Academic Publishers, 2001.

\bibitem{ioffepenot} Ioffe A.D., Penot J.-P., Limiting Subhessians, Limiting
Subjects and Their Calculus // Transactions of the American
Mathematical Society. 349(2), 789-807 (1997)

\bibitem{kruger} Kruger A. Ya. On Frechet subdifferentials // J. of Math.
Sciences.N.Y., 2003. V. 116, No. 3. P. 3325-3358.

\bibitem{pshenichnyi} Pshenichnyi B.N. Convex Analysis and Extremal Problems. Moscow:
Nauka, 1980, 320 p.

\bibitem{penot} Penot J.P. Subhessians, Superhessians, and Conjugation // Nonlinear
Anal. 23, 689-702 (1994)

\bibitem{pimnewconstr1} Proudnikov I.M. New Constructions for Local Approximation of
Lipschitz Functions. I, Nonlinear Analysis, Vol. 53, No. 3, 2003,
pp. 373-390.

\bibitem{proudexhaust} Prudnikov I.M.  A Method for constructing an
Exhaustive set of upper onvex approximations. // Vestn. St.
Petersburg University. Ser. 10. 2013. Issue 1. pp. 37-51.

\bibitem{pimfirstsecondsubdiff} Prudnikov I.M. The subdifferentials
of the first and second orders for Lipschitz functions // J. of
Optimization Theory and Application. 171(3). 906-930. (2016)

\bibitem{pimTwiceContinCodif} Prudnikov I.M., Duxbury N.S. Conditions
on Twice Continuous Codifferentiability of Quasidifferentiable
Functions // J. of Optimization Theory and Application. 2025.
https://doi.org/10.1007/s10957-024-02549-5

\bibitem{roshchina} Roshchina V. On the relationship between the Frechet
subdifferential and upper exhausters // International Workshop on
Optimization: Theory and Algorithms. 19-22 August 2006.
Zhangjiajie, Hunan, China.



\end{thebibliography}
\end{document}